\documentclass[journal]{IEEEtran}

\usepackage{amssymb,amsmath,amsthm,amscd}
\usepackage{amsfonts}
\usepackage{color}
\usepackage{soul}
\usepackage{cancel}
\usepackage{hyperref}

\begin{document}
\newtheorem{theorem}{Theorem}
\newtheorem{corollary}{Corollary}
\newtheorem{example}{Example}
\newtheorem{main}{Main Theorem}
\newtheorem{lemma}{Lemma}
\newtheorem{proposition}{Proposition}
\newtheorem{conjecture}{Conjecture}
\newtheorem*{problem}{Problem}
\theoremstyle{definition}
\newtheorem{definition}{Definition}
\newtheorem{remark}{Remark}
\newtheorem*{notation}{Notation}
\newtheorem{manualtheoreminner}{Theorem}

\newcommand{\Tr}{\mathrm{Tr}}
\newcommand{\ZZ}{\mathbb{Z}}
\newcommand{\RR}{\mathcal{R}}
\newcommand{\FF}{\mathbb{F}}
\newcommand{\TT}{\mathcal{T}}
\newcommand\GF{\mathrm{GF}}
\newcommand\GR{\mathrm{GR}}
\newcommand\CCo{\mathrm{Cos}}
\newcommand\Cos{\mathrm{Cos}}
\newcommand\BCCo{\mathrm{BCCo}}
\newenvironment{manualtheorem}[1]{%
  \renewcommand\themanualtheoreminner{#1}%
  \manualtheoreminner
}{\endmanualtheoreminner}
\renewcommand\baselinestretch{1.2}

\title{Quasi-cyclic perfect codes in Doob graphs and special partitions of Galois rings%
\thanks{This is the accepted peer-reviewed version of the paper published in the IEEE
Transactions on Information Theory, 
\url{https://doi.org/10.1109/TIT.2023.3272566}
\copyright~2023 IEEE. Personal use is permitted, but
republication/redistribution requires IEEE permission}%
\thanks{This research is supported by Natural Science Foundation of China (12071001).
The work of D.\,S.\,Krotov was carried out within the framework of the state contract
of the Sobolev Institute of Mathematics (Project FWNF-2022-0017).}%
}

\author{%
Minjia~Shi,
Xiaoxiao~Li,
Denis~S.~Krotov, 
Ferruh~{\"O}zbudak
\thanks{Minjia Shi is with the Key Laboratory of Intelligent Computing and Signal Processing, Ministry of Education, School of Mathematical Sciences, Anhui University, Hefei, 230601, China; State Key Laboratory of Information Security Institute of Information Engineering, Chinese Academy of Sciences, Beijing, 100093.}
\thanks{Xiaoxiao Li is with School of Mathematical Sciences, Anhui University, Hefei, 230601, China.}
\thanks{Denis S. Krotov is with Sobolev Institute of Mathematics, pr. Akademika Koptyuga 4, Novosibirsk 630090, Russia (e-mail: krotov@math.nsc.ru).}
\thanks{Ferruh \"{O}zbudak is affiliated with Faculty of Engineering and Natural Sciences, Sabanc{\i} University, 34956, Istanbul, and  Middle East Technical University, 06800, Ankara, Turkey (e-mail: ferruh.ozbudak@sabanciuniv.edu).}}

\maketitle

\begin{abstract}   \boldmath
The Galois ring GR$(4^\Delta)$ is the residue ring $Z_4[x]/(h(x))$, where $h(x)$ is a basic primitive polynomial of degree $\Delta$ over $Z_4$. For any odd $\Delta$ larger than $1$, we construct a partition of GR$(4^\Delta) \backslash \{0\}$ into $6$-subsets of type $\{a,b,-a-b,-a,-b,a+b\}$ and $3$-subsets of type $\{c,-c,2c\}$ such that the partition is invariant under the multiplication by a nonzero element of the Teichmuller set in GR$(4^\Delta)$ and, if $\Delta$ is not a multiple of $3$, under the action of the automorphism group of GR$(4^\Delta)$.

As a corollary, this implies the existence of quasi-cyclic additive $1$-perfect codes of index $(2^\Delta-1)$ in $D((2^\Delta-1)(2^\Delta-2)/{6}, 2^\Delta-1 )$ where $D(m,n)$ is the Doob metric scheme on $Z^{2m+n}$.
\end{abstract}


%
\IEEEpeerreviewmaketitle

\section{Introduction}
%
%
%
%
\IEEEPARstart{T}{he} motivation of studying the codes in {Doob graph} is the application of
association schemes in coding theory \cite{Delsarte:1973}.
Codes in {Doob graphs} can be used for information transmission in channels with
two-dimensional or complex-valued modulation. In fact, these codes are special
cases of codes over Eisenstein--Jacobi integers, see, e.g., \cite{Huber94},
\cite{MSBG:2008}. Form the point of view of the distance-regular graphs, the
algebraic parameters of the schemes associated with Doob graphs are the same as
for the quaternary Hamming scheme.
The vertices of a {Doob graph} can be
considered as words in the mixed alphabet consisting of elements of the quotient
(modulo $4$ and modulo $2$) rings of the ring of Eisenstein--Jacobi integers, see e.g. \cite{Kro:perfect-doob}.

The \emph{Doob graph} $D(m,n)$
is the Cartesian product of $m$ copies of the Shrikhande graph~$\mathrm{Sh}$ and~$n$ copies of the complete graph $K_4$ on $4$ vertices.
The \emph{Shrikhande graph} is the Cayley graph
on $\ZZ_4^2$ with the connecting set
$\{\pm(0,1),\pm(1,0),\pm(1,1)\}$,
where two vertices are adjacent if and only if their difference is in the connecting set.
The complete graph~$K_4$ in our theory
is a Cayley graph on~$\ZZ_4$ or~$\ZZ_2^2$.
A set of vertices of a graph is called
a \emph{$1$-perfect code}
if every vertex is at distance at most~$1$ from exactly one code vertex.

The study of perfect codes in Doob graphs was initiated by Koolen and Munemasa in
\cite{KoolMun:2000}, where they constructed $1$-perfect codes
in the two Doob graphs of diameter $5$, $D(1,3)$ and $D(2,1)$.
In~\cite{Kro:perfect-doob}, linear $1$-perfect codes are characterized,
where a code is linear if it is a submodule of the module
$\GR(4^2)^m \times \GF(2^2)^n$, naturally associated with the vertex set
of $D(n,m)$; such codes can be constructed if and only if
\begin{equation}
 \label{eq:GaDe}
n=\frac{4^{\gamma+\delta}-1}3,\qquad
m=\frac{4^{\gamma+2\delta}-4^{\gamma+\delta}}6
\end{equation}
for some integers $\gamma\ge0$ and $\delta>0$ (case $\delta=0$ also has a sense,
but it corresponds to the Hamming graph $D(0,n)=H(n,4)$ over the $4$-ary alphabet).
However, the sphere-packing necessary condition on the existence of
$1$-perfect codes requires only the divisibility of the number $16^m\cdot 4^n$ of the vertices of the graph by the cardinality $1+6m+3n$ of a radi\-us\mbox{-$1$} ball.
The family of parameters satisfying this condition is much wider than the ones meeting~\eqref{eq:GaDe}.
For example, all $10$ Doob graphs
of diameter $21$,
$D(1,19)$, $D(2,17)$, \ldots , $D(10,1)$, have radius-$1$
balls of size $64$, dividing $4^{21}$,
while only
one of them, $D(8,5)$, falls into restrictions~\eqref{eq:GaDe}.
As was shown in~\cite{Kro:Doob-exist},
the sphere-packing
divisibility
condition is sufficient for the existence
of unrestricted $1$-perfect codes in a Doob graph,
where ``unrestricted'' means that the code is not implied
to have any algebraic structure.
Motivated by a possibility to cover by algebraic constructions more parameters than the parameters
of linear codes, the concept of additive
codes was suggested in~\cite{Kro:perfect-doob}.
A Doob graph $D(m,n)$ inherits the structure 
of the module $\ZZ_4^{2m+n''}\times \ZZ_2^{2n'}$,
where $n=n'+n''$, and a code is additive if it is additively closed, or, equivalently, forms a submodule of this module.
Necessary conditions
\begin{multline}
\label{eq:addt}
2m+n'+n'' = \frac{2^{\Gamma+2\Delta}-1}3,
\quad
  3n'+n'' = 2^{\Gamma+\Delta}-1,
  \\
 0\leq n'' \leq  2^{\Delta}-1,\quad
  n''\ne 1, \quad   \Gamma\ge0, \quad
  \Delta\ge 2
\end{multline}
were given and first examples
constructed in~\cite{Kro:perfect-doob},
while constructions developed in~\cite{SHK:additive}
cover the whole three-parameter family of Doob graphs
$D(m,n'+n'')$ satisfying \eqref{eq:addt}.
Apart from recursive constructions solving the existence problem for
additive $1$-perfect codes in Doob graphs,
three nice codes were described in \cite{Kro:perfect-doob}, \cite{SHK:additive}.
They are in the graphs
\begin{equation}\label{eq:D00}
D\Big((2^{\Delta}-1)\cdot\frac{2^{\Delta}-2}{6},\  0+(2^{\Delta}-1)\Big)
\end{equation}
with $\Delta=3,5,7$ and have a quasicyclic permutational
automorphism of order~$2^\Delta - 1$.
Generalizing these quasicyclic additive $1$-perfect codes to the case of an arbitrary odd~$\Delta$
is the main goal of the current paper.

Let us describe why the parameters \eqref{eq:D00}
are special. An additive $1$-perfect code~$C$
is a submodule of $\ZZ_4^{2m+n''}\times \ZZ_2^{2n'}$,
and the quotient
$Q = (\ZZ_4^{2m+n''}\times \ZZ_2^{2n'})/ C$ is  isomorphic to $\ZZ_4^\Delta\times \ZZ_2^\Gamma$.
It is a free module if $\Gamma = 0$.
In this case,
$Q$ can be associated with the Galois ring
$\GR(4^\Delta)$, and $C$ can potentially have a coordinate-permutation
automorphism of order~$2^\Delta-1$ associated to the multiplication by
a {nonzero element of the Teichmuller set}
in~$\GR(4^\Delta)$.
This can only happen if the numbers
$m$, $n'$, $n''$
of coordinates
of each type are divisible by~$2^\Delta-1$,
which immediately leads to~$n'=0$ and
to the parameters shown in~\eqref{eq:D00}.

Let us consider several motivation points
of our study. At first, cyclic and quasicyclic codes are
of theoretical and practical importance
because of the possibility
to represent the code in a compact way and,
as a corollary, to create coding and decoding algorithms that are efficiently realized in hardware.
Codes in mixed alphabets seem to be rarely
used in practice.
However, the number~$n''$
of the last-group of ``Hamming'' coordinates
is relatively small with respect to the
number~$m$
of ``Shrikhande'' coordinates, and shortening a
$1$-perfect code in all these $n''$~coordinates
makes it a non-mixed code in~$D(m,0)$,
non-perfect, but still very close to optimal.
We also mention the result
of
\cite{BorFer:1cyclic} which states
that cyclic $1$-perfect codes
over the mixed $\ZZ_2\ZZ_4$-alphabet with the Lee metric do not exist, except for one example. Taking into account some similarity between codes over
the mixed $\ZZ_2\ZZ_4$-alphabet
and codes in Doob metric,
we can conjecture that cyclic $1$-perfect
codes hardly exist in Doob graphs,
and the quasicyclic structure
is the best possible.

At second, additive codes we consider are related with
a very interesting class of partitions
of a module over~$\ZZ_4$.
To construct a check matrix for an additive
$1$-perfect code in $D(m,n'+n'')$,
one needs to partition the non-zero elements of
$\ZZ_4^\Delta \times \FF_2^\Gamma$
into $m$~sixtuples of type $\{\pm a, \pm b, \pm(a+b) \}$,
$n''$~triples of type $\{ a, 2a, -a \}$,
and $n'$~triples of type $\{2a, 2b, 2a+2b \}$
(the last group is empty in our current research).
In the case $m=n''=\Delta = 0$, such partitions
are a special case of \emph{spreads}, widely studied in discrete geometry, see e.g.~\cite{Johnson:spreads}.
For all admissible collections
of parameters $(m,n',n'')$,
required paritions were constructed
in~\cite{SHK:additive},
using three combinatorial recursive constructions.
That technique is useless
if we want to construct quasi-cyclic
perfect codes with
the corresponding partitions
invariant under the multiplicative action
of the Teichmuller set in~$\GR(4^\Delta)$.
So, we use totally different
algebraic technique to prove
the existence of partitions with this
nice property (moreover, in the case
$\Delta\not\equiv 0\bmod 3$, our partitions are also invariant under the action of the automorphism group of~$\GR(4^\Delta)$).
Note that predescribing an automorphism
group, often based on multiplications in the field or the ring associated to the ambient space, is an effective approach
for finding different combinatorial
structures computationally,
see e.g.~\cite{BEOVW:q-Steiner},
where discovering of the first known
subspace Steiner structures is described.
In our study, we also began with
solving small cases computationally,
but finally found the theoretical approach,
presented in this paper.

Finally, as was noted in~\cite{KoolMun:2000}, the codes with the parameters dual to the parameters of $1$-perfect codes
are singleton codes, also known
as tight $2$-designs,
and with properly defined inner product
and dual graph
(the details can be found in~\cite{Krotov:ISIT2019:Doob}),
additive $1$-perfect codes in Doob graphs
correspond to additive singleton codes.
Moreover, the duality preserves
permutational automorphisms, which means that we construct quasicyclic singleton codes as well.

The structure of the paper is as follows. The next section compiles the background necessary to the forthcoming sections. Special partitions of Galois rings are constructed in Section~\ref{s:part}.
Section~\ref{s4} contains the main result of this paper: a construction
of quasicyclic $1$-perfect codes in Doob graphs. Section~\ref{s5} concludes the paper.

\section{Background material}\label{s2}
In this section, we introduce some preliminary results from three parts as follows.
\subsection{Galois Rings}\label{s2.1}
Let $\Delta$ be odd. If $h(x)$ is a basic primitive polynomial of degree $\Delta$ over~$\ZZ_4$ and $\xi$ is a root of~$h(x)$, then any element in the residue class ring $\ZZ_4[x]/(h(x))$ can be written as $h_0+h_1\xi+\cdots+h_{\Delta-1}\xi^{\Delta-1}$, which could also be viewed as the vector $(h_0,h_1,\ldots,h_{\Delta-1})$ over $\ZZ_4$.
The ring $\ZZ_4[x]/(h(x))$ is called the Galois ring with $4^\Delta$ elements and is denoted by~$\GR(4^\Delta)$.
The Teichmuller set $\TT=\{0,1,\xi,\xi^2,\ldots,\xi^{2^\Delta-2}\}$ is a set of representatives of the residue field  $\FF_{2^\Delta}\simeq \GR(4^\Delta)/2\GR(4^\Delta)$. It is known that $\GR(4^\Delta)=\TT\oplus 2\TT$ (2-adic decomposition of $\GR(4^\Delta)$), and that the group of units of the Galois ring is $\GR(4^\Delta)^*=\TT^*\oplus 2\TT$ with $\TT^*=\TT\setminus \{0\}$.
The following formula for the sum of elements of $\TT$ is useful.
\begin{lemma}[Yamada's formula, see, e.g.,~{\cite[Corol.\,6.9]{Wan:4ary}}]\label{Yamada}
If $c_1$, $c_2\in \TT$, and $c_1+c_2=a+2b$, $a,b \in \TT,$
then
$$a=c_1+c_2+2(c_1c_2)^{1/2},b=(c_1c_2)^{1/2},$$
where $(c_1c_2)^{1/2}$ denotes the unique element in~$\TT$ such that $((c_1c_2)^{1/2})^2=c_1c_2$.
\end{lemma}

We know that the kernel of the ring homomorphism
\begin{IEEEeqnarray*}{rCl}
\overline{\phantom a}: \ZZ_4[x] &\to & \ZZ_2[x] \\
\overline{h_0+h_1x+\cdots+h_nx^n} &=& \overline{h}_0+\overline{h}_1x+\cdots+\overline{h}_nx^n
\end{IEEEeqnarray*}
is the ideal
$2\ZZ_4[x]$
and the image of~$(h(x))$ under $\overline{\phantom a}$ is $(\overline{h}(x))$. Therefore the ring homomorphism induces a ring homomorphism
\begin{IEEEeqnarray*}{rCl}
\overline{\phantom a}: \ZZ_4[\xi] &\to & \ZZ_2[\overline{\xi}] \\
\overline{h_0+h_1\xi+\cdots+h_{\Delta-1}\xi^{\Delta-1}} &=& \overline{h}_0+\overline{h}_1\overline{\xi}+\cdots+\overline{h}_{\Delta-1}\overline{\xi}^{\Delta-1},
\end{IEEEeqnarray*}
where $\overline{\xi}$ is the image of~$\xi$. Actually, $\overline{\xi}$ is a root of~$(\overline{h}(x))$ and $\ZZ_2[\overline{\xi}]=\FF_{2^\Delta}$.

The generalized Frobenius map on $\GR(4^\Delta)$ defined by
$$f:\GR(4^\Delta)\rightarrow \GR(4^\Delta),
\quad
c=a+2b\mapsto f(c)=a^2+2b^2,
$$
where $a,b \in \TT$, is a ring automorphism of $\GR(4^\Delta)$. Moreover, if $\sigma$ is a ring automorphism of~$\GR(4^\Delta)$, then $\sigma=f^l$ for some $l$, $0\leq l\leq \Delta-1$, see e.g.~\cite[Chapter $6$]{Wan:4ary}.

\subsection{Trace map}
For $u\in \FF_{2^\Delta}$, the trace $\mathrm{Tr}_{\FF_{2^\Delta}/ \FF_{2}}(u)$ of~$u$ over~$\FF_{2}$ is defined by
$$\Tr_{\FF_{2^\Delta}/ \FF_{2}}(u)=u+u^2+\cdots+u^{2^{\Delta-1}}.$$
$\Tr_{\FF_{2^\Delta}/ \FF_{2}}(u)$ is the absolute trace of $u$ and simply denoted by~$\Tr(u)$.

The following lemmas are important
for us to prove the main theorems.

\begin{lemma}\label{solution1}
If $y_1,y_2\in \FF_{2^\Delta}$ and $\Tr(y_1)=\Tr(y_2)$, then $\Tr(y_1y_2)=0$ or $\Tr((y_1+1)(y_2+1))=0$.
\end{lemma}
\begin{proof}
It can be directly calculated by $\Tr((y_1+1)(y_2+1))=\Tr(y_1y_2)+\Tr(y_1+y_2)+\Tr(1)=\Tr(y_1y_2)+1$.
\end{proof}
The next fact is a special case of a well-known result, Hilbert's Theorem 90 (see, \cite[Note~3 in page~75]{LN}). We give a simple proof for the special case that we consider.
\begin{lemma}\label{solution2}
If $a\in \FF_{2^\Delta}$, then the equation
\begin{equation}\label{eq1}
 x^2+x = a
\end{equation}
has two solutions in $\FF_{2^\Delta}$ if $\Tr(a)=0$
and has no solutions  in $\FF_{2^\Delta}$ if $\Tr(a)=1$.
\end{lemma}
\begin{proof}
The claim is straightforward from the following two observations.
\begin{enumerate}
\item [(i)] For every element~$x_0$ of $\FF_{2^\Delta}$,
both~$x_0$ and~$x_0+1$ are solutions of~(\ref{eq1}) with
$a=x_0(x_0+1)$.
On the other hand, a quadratic equation cannot have more than two solutions.
This means that (\ref{eq1}) is solvable for $2^\Delta / 2$ values of~$a$.

\item [(ii)] If $x_0$ is a solution of (\ref{eq1}), then $\Tr(a)=\Tr(x_0)+\Tr(x_0^2) =
\Tr(x_0)+\Tr(x_0)=0$.
\end{enumerate}
\end{proof}
\subsection[Additive 1-perfect codes in Doob graphs]{Additive 1-perfect codes in Doob graphs}
Denote by $D(m, n)$ the Cartesian product $\mathrm{Sh}^m\times K^n$ of $m$ copies of the Shrikhande graph $\mathrm{Sh}$ and $n$ copies of the complete $4$-vertex graph.
If $m > 0$, then $D(m, n)$ is called a \emph{Doob graph}.
(case $m = 0$ corresponds to a Hamming graph $H(n,4)=D(0,n)$).
The \emph{Shrikhande graph} $\mathrm{Sh}$ is the Cayley graph of the additive group $\ZZ^{2+}_4$ of $\ZZ^2_4$ with the connecting set $S = \{01, 30, 33, 03, 10, 11\}$.
Take the set of $(2m +2n)$-tuples
$(x_1, \ldots, x_{2m},z_1, \ldots,z_{n} )$ from $\ZZ^{2m}_4 \times \ZZ^{n}_4$
as the vertex set of $D(m, n)$.
If a code
$C \subset \ZZ^{2m}_4 \times \ZZ^{n}_4$
is closed with respect to addition,
then we say it is \emph{additive}.
An additive code is necessarily
closed with respect to multiplication by an element of
$\ZZ_4$.
So, it is in fact a submodule of
$\ZZ^{2m+n}_4$.

\begin{remark}
 In~\cite{Kro:perfect-doob}, additive codes
 are defined as a more general class:
 some of the last $n$ coordinates
 have values from $\ZZ_2^2$.
 However, as explained in the introduction,
 in our current study such coordinates are absent.
\end{remark}

A $1$-perfect code in a Doob graph $D(m,n)$
is a subset C of
$\ZZ_4^{2m}\times \ZZ_4^{n}$
which is an independent set such that
every vertex in $\ZZ^{2m}_4\times\ZZ^{n}_4\setminus C$
is adjacent to exactly one vertex in $C$.

Let $(A | A'')$ be a matrix
with rows from $\ZZ^{2m}_4\times\ZZ^{n}_4$.
It is called a \emph{check matrix}
of a code $C$ in $D(m, n' + n'')$ if
$C = \{c \in \ZZ^{2m}_4 \times \ZZ^{n}_4 :\  (A | A'' )c^{T} = \overline{0}^{T}\}$.


\begin{lemma}[{\cite[Lemma~1]{SHK:additive}}]\label{l:Aperf}
Let $(A |A'')$
be a check matrix of an additive
code $C$ in $\ZZ_4^{2m}\times \ZZ_4^{n}$.
Denote by $a_i$ the $i$th column of $(A |A'')$, and by $Q$ the linear span of all
$a_i$,
$1\le i \le 2m+n$.
The code $C$ is $1$-perfect in $D(m,n)$
if and only if the collection of the following sets
form a partition of $Q\backslash\{\overline0^T\}$:
\begin{itemize}
 \item[\rm(I)] six-element sets
 $$\{\pm a_{2i-1}, \pm a_{2i}, \pm (a_{2i-1}+a_{2i}) \},
 \quad i=1, \ldots , m,$$
 \item[\rm(II)] three-element sets
 $$\{a_{j}, 2a_{i}, -a_j \}, \quad j=2m+1, \ldots , 2m+n.$$
\end{itemize}
\end{lemma}

In general, $Q$ is a module of type
$\ZZ_4^\Delta\times\ZZ_2^\Gamma$ fore some $\Delta$ and $\Gamma$,
but we are interested in the case $\Gamma =0$, and assume
for the convenience that $Q=\GR(4^\Delta)$.
Our goal is to find a partition of
$\GR(4^\Delta)\backslash\{0^\Delta\}$ into sets
of type (I) and (II) such that the partition is invariant
under the multiplication by $\xi$.
We set $a_j=\xi^j$, $j=2m+1, \ldots , 2m+n$, so the elements
of form $\xi^j$, $2\xi^j$, and $3\xi^j = -\xi^j$ are covered
by the sets of type (II). Then, the remaining elements
of $\GR(4^\Delta)$ have to be covered by sets of type (I).

\section[Partitions of GR(...)]{Partitions of $\GR(4^\Delta)$}\label{s:part}
Let $\RR=\GR(4^\Delta)$. Every element of~$\RR$ is uniquely represented in the form $A=X+2Y$, where $X,Y\in \TT$. If $A\not\in\TT\cup2\TT$, then $A=\xi^i(1+2\xi^j)$ for some unique~$i$ and~$j$ from
$\{0, \ldots  , 2^\Delta-2\}$. These will be denoted
by~$i(A)$ and~$j(A)$, respectively.
Let $C_i=\{i\cdot2^j \bmod 2^\Delta-1 : 0\leq j\leq \Delta-1\}$, $0\leq i\leq 2^\Delta-2$. For each $i$, the set $C_i$ is called a \emph{$2$-cyclotomic coset} modulo $2^\Delta-1$.
It is straightforward that
$C_i =C_j$ or $C_i\cap C_j=\emptyset$ for any~$i$,~$j$.
By~$\CCo$, we denote the set of all $2$-cyclotomic cosets except~$\{0\}$;
by~$\CCo(s)$, the set
$\{\mathcal{A}\in \CCo: |\mathcal{A}|=s\}$,
and by~$N_s$, its cardinality.

\begin{lemma}\label{j(-A)}
For every $A$, $B$ from $\RR\backslash
(\TT \cup 2\TT \cup 3\TT)$,
\begin{itemize}
 \item[\rm(i)] if $j(A)=j(B)$, then $j(-A)=j(-B)$;
 \item[\rm(ii)] $j(A)$ and $j(-A)$ ($j(B)$ and $j(-B)$) belong to different cosets  in $\CCo$;
 \item[\rm(iii)] if $j(A)$ and $j(B)$ are in the same coset, then
 $j(-A)$ and $j(-B)$ are in the same coset.
\end{itemize}
\end{lemma}
\begin{proof}
Without losing generality, suppose $A=1+2\xi^{j(A)}$ and $B=\xi^{i(B)}(1+2\xi^{j(B)})$.
We have
\begin{IEEEeqnarray}{rCl}
\xi^{i(-B)}(1+2\xi^{j(-B)})-B &=& B+2B \nonumber \\
&=& \xi^{i(B)}\big(1+2(1+\xi^{j(B)})\big) \label{eq:-B} \\
&=& \xi^{i(B)}\big(1+2\xi^{j}\big),\nonumber
\end{IEEEeqnarray}
for some~$j$.
Therefore, $j=j(-B)$ and $i(B)=i(-B)$. Similarly,
$i(A)=i(-A)=0$.

(i) If $j(A)=j(B)$, then
$B = \xi^{i(B)} A$, which implies
$-B = -\xi^{i(B)} A$, $-B = \xi^{i(-B)} (- A)$,
and $j(-B)=j(-A)$.

(ii) From~\eqref{eq:-B} we see that $2(1+\xi^{j(B)})=2\xi^{j(-B)}$.
Hence,
$1+\overline{\xi}^{j(B)}=\overline{\xi}^{j(-B)}$,
and $1+\Tr(\overline{\xi}^{j(B)})=\Tr(\overline{\xi}^{j(-B)})$.
On the other hand,
$\Tr(\overline{\xi}^{j(B)}) = \Tr(\overline{\xi}^{2^s j(B)})$
for all $s$.
Hence, there is no $s$ such that $j(B) \equiv 2^s j(-B) \bmod 2^\Delta-1$,
which implies claim~(ii) for~$j(\pm B)$. For~$j(\pm A)$, the proof is similar.

(iii)
If $j(A)$ and $j(B)$ are in the same coset, then
for some $s$ we have
$\overline{\xi}^{j(B)} = \overline{\xi}^{2^s j(A)}$.
It follows that
\begin{IEEEeqnarray*}{rCl}
\overline{\xi}^{j(-B)} &=& 1+\overline{\xi}^{j(B)} =
1^{2^s}+\big(\overline{\xi}^{j(A)}\big)^{2^s} \\ &=&
\big(1+\overline{\xi}^{j(A)}\big)^{2^s}  =
\big( \overline{\xi}^{j(-A)} \big)^{2^s}
=
\overline{\xi}^{2^sj(-A)},
\end{IEEEeqnarray*}
which implies that~$j(-A)$ and~$j(-B)$ are in the same coset as well.
\end{proof}
Lemma~\ref{j(-A)} establishes a matching on the cosets in~$\CCo$.
For a coset~$\mathcal{A}$,
the coset matched to~$\mathcal{A}$ will be denoted
 by~$-\mathcal{A}$. The notation~$\pm \mathcal{A}$
 will mean the union of~$\mathcal{A}$ and~$-\mathcal{A}$.

\begin{proposition}\label{N_s}
 For every positive integer~$\Delta$
 and every integer~$s>2$, the value $N_s\cdot s$ is divisible by~$3$.
\end{proposition}
\begin{proof}
Assume $a$ is an element of a
$2$-cyclotomic coset
of size $s$ in the cyclic group~%
$\ZZ_{2^\Delta-1}^+$.
By the definition, we have
\begin{equation}\label{eq:sD}
 2^s a \equiv a \bmod 2^\Delta-1,
\end{equation}
where $s$ is the minimum
positive integer for which
\eqref{eq:sD} holds.
Since, obviously,
$2^\Delta a \equiv a \bmod 2^\Delta-1$,
we see that $s$ divides~$\Delta$.
In particular,
$2^s-1$ divides
$2^\Delta-1$.
Next, \eqref{eq:sD}
means that
$
(2^s-1) a
=
A (2^\Delta-1)
$ for some $A$.
Hence,
$ a
=
A \cdot \frac{2^\Delta-1}{2^s-1}$,
and $a$ is a multiple of
$\frac{2^\Delta-1}{2^s-1}$.
We conclude the following:
\begin{itemize}
 \item [(*)] \it
 Every coset of size $s$
 lies in the subgroup
 $ \frac{2^\Delta-1}{2^s-1} \ZZ_{2^\Delta-1}^+$ of order ${2^s-1}$
 of the group~$ \ZZ_{2^\Delta-1}^+$.
\end{itemize}
With (*), the claim of the proposition can be easily proved by induction on $\Delta$. The induction base is
$\Delta = 1$ and $\Delta = 2$,
when there are no $s$ satisfying
$s>2$ and $N_s>0$.
Assume that $\Delta > 2$.
From (*) and the induction
hypothesis, we have the following:
\begin{itemize}
 \item [(**)] \it
 If $2<s<\Delta$,
 then
 $N_s\cdot s$ is divisible by $3$.
\end{itemize}
It remains to show
that $N_\Delta \cdot \Delta$
is divisible by $3$.
We derive this from the obvious
$2^\Delta-1
=|\ZZ_{2^\Delta-1}^+|
 =\sum_{s\le\Delta}
 N_s\cdot s$:
 \begin{IEEEeqnarray*}{rCl}
   N_\Delta \cdot \Delta &=& 2^\Delta-1 - \sum_{s<\Delta}
 N_s\cdot s \\
    &=& \big(2^\Delta-1 - N_1\cdot 1
 - N_2\cdot 2\big)
 - \sum_{s:\,2<s<\Delta}
 N_s\cdot s.
 \end{IEEEeqnarray*}
Since $N_1=1$,
$N_2 = 1$ for even~$\Delta$,
and
$N_2 = 0$ for odd~$\Delta$,
the expression in the parenthesis is divisible by~$3$.
The remaining part is divisible
by~$3$ according to (**).
Hence, $N_\Delta \cdot \Delta$
is divisible by $3$, and the proof by induction is complete.
\end{proof}

\begin{theorem}\label{nd3}
Assume that $\Delta$ and $s$
are odd positive
integers such that
$s$ divides $\Delta$, $s>2$, and
$s$ is not divisible
by~$3$.
In
$\RR\backslash
(\TT \cup 2\TT \cup 3\TT)$,
there are
$A_1$, $B_1$, \ldots, $A_k$, $B_k$,
where $k=N_s/6$,
such that $j(A_i)$, $j(B_i)$, $j(-A_i-B_i)$,
$j(-A_i)$, $j(-B_i)$, $j(A_i+B_i)$, $i=1,\ldots,k$, represent all cosets of size $s$ in $\CCo$; in particular,
$|C_{j(A_i)}|=|C_{j(B_i)}|=|C_{j(-A_i-B_i)}|=
|C_{j(-A_i)}|=|C_{j(-B_i)}|=|C_{j(A_i+B_i)}| = s$ for each $i=1,\ldots,k$.
\end{theorem}

\begin{proof}
 Since $s$ is not divisible by $3$, the size $N_s$ of $\CCo(s)$ is divisible by $3$ by Proposition~\ref{N_s}.
According to case (ii) of Lemma~\ref{j(-A)},
if $\mathcal{A}\in \CCo(s)$ then $-\mathcal{A}\in \CCo(s)$, i.e, $N_s$ is divisible by $2$.
Hence, $N_s$ is divisible by $6$.

By Lemma~\ref{j(-A)},
it is sufficient to prove the following:
\begin{itemize}
 \item[(*)] \emph{For any given three cosets
$\mathcal{A}$, $\mathcal{B}$,
$\mathcal{C}\in \CCo(s)$ such that
$   \pm \mathcal{A}
\ne \pm \mathcal{B}
\ne \pm \mathcal{C}
\ne \pm \mathcal{A}
$,
there exist
$A$ and $B$ in
$\RR\backslash
(\TT \cup 2\TT \cup 3\TT)$ such that
$$j(A)\in \mathcal{A},
\quad
j(B) \in \pm \mathcal{B},
\quad
j(A+B)\in \pm \mathcal{C}.
$$
}
\end{itemize}

We now assume $i(A)=1$ without loss of generality,
fix arbitrary values of
$j(A)$ in~$\mathcal A$,
$j(B)$ in~$\mathcal B$, and
$j'$ in~$\mathcal C$, and try to find
the value of $i(B)$ such that $j(A+B)=j'$.
Denote $\xi^{j(A)}$, $\xi^{i(B)}$, $\xi^{j(B)}$, $\xi^{j'}$ by $\alpha$, ${x}$, $\beta$, $\gamma$ respectively; in particular,
$$A=1+2\alpha \quad \mbox{and} \quad B={x}(1+2\beta).$$
From Yamada's formula (Lemma~\ref{Yamada}),
we have
\begin{equation}\label{eq:A+B}
A+B=(1+{x}+2\sqrt{{x}})(1+2\gamma),
\end{equation}
where $ (1+{x}+2\sqrt{{x}}) \in \TT. $
Recall that we want to find~${x}$.
After expanding, \eqref{eq:A+B} turns to
$$2(1+{x})\gamma=2(\alpha+{x}\beta+\sqrt{{x}}),$$
which is equivalent to
$$(1+\overline{x})\overline{\gamma}=\overline{\alpha}+\overline{x}\overline{\beta}+\sqrt{\overline{x}}.$$
Multiplying both sides by $\overline{\beta}+\overline{\gamma}$ and denoting $T=(\overline{\beta}+\overline{\gamma})\sqrt{\overline{x}}$ we rewrite it as
\begin{equation}\label{eq:T2T}
T^2+T+(\overline{\alpha}+\overline{\gamma})(\overline{\beta}+\overline{\gamma})=0
\end{equation}
(note that $ \overline{\beta}+\overline{\gamma} \ne 0$ because $j(B)$ and~$j'$ belong to different cosets, by the hypothesis).
We now consider two cases.

\emph{Case 1.} Assume $\Tr\big((\overline{\alpha}+\overline{\gamma})(\overline{\beta}+\overline{\gamma})\big)=0$.
By Lemma~\ref{solution2},
there is a nonzero solution~$T$ of equation~\eqref{eq:T2T}.
From~$T$, we find $\overline{x} = T^2/(\overline{\beta}+\overline{\gamma})^2$,
$x$, and $i(B)=\log_{\overline{\xi}}\overline{x}=\log_{\xi}{x}$.
From~\eqref{eq:A+B} we see that $j(A+B)=j'$, as required, and (*) holds with
$$j(A)\in \mathcal{A},
\quad
j(B) \in  \mathcal{B},
\quad
j(A+B)\in \mathcal{C}.
$$

\emph{Case 2.} Assume  $\Tr\big((\overline{\alpha}+\overline{\gamma})(\overline{\beta}+\overline{\gamma})\big)=1$.
Depending on the values of $\Tr(\overline{\alpha}+\overline{\gamma}) $ and
$\Tr(\overline{\beta}+\overline{\gamma})$,
at least one of the following is true:
\begin{IEEEeqnarray}{rCl}
 \Tr\big((\overline{\alpha}+\overline{\gamma})(\overline{\beta}+\overline{\gamma}+1)\big)&=&0, \label{eq:01}\\
\Tr\big((\overline{\alpha}+\overline{\gamma}+1)(\overline{\beta}+\overline{\gamma})\big)&=&0, \label{eq:10}\\
\Tr\big((\overline{\alpha}+\overline{\gamma}+1)(\overline{\beta}+\overline{\gamma}+1)\big)&=&0 \label{eq:11}.
\end{IEEEeqnarray}
In  subcase \eqref{eq:11},
replacing $\overline{\gamma}$ by $\overline{\gamma}+1$
(which changes~$j'$ to~$\log_{\overline{\xi}}(\overline{\xi}^{j'}+1)$
and moves it to $-\mathcal{C}$)
makes equation~\eqref{eq:T2T} solvable and proves~(*) with
$$j(A)\in \mathcal{A},
\quad
j(B) \in  \mathcal{B},
\quad
j(A+B)\in -\mathcal{C}.
$$
In subcase~\eqref{eq:01},
we replace $\overline{\beta}$ by $\overline{\beta}+1$
(which changes the choice of $j(B)$ and
  moves it to $-\mathcal{B}$)
makes the equation solvable and proves~(*) with
$$j(A)\in \mathcal{A},
\quad
j(B) \in  -\mathcal{B},
\quad
j(A+B)\in \mathcal{C}.
$$
Subcase~\eqref{eq:10} is solved similarly with replacing
both $\overline{\beta}$ by $\overline{\beta}+1$ and $\overline{\gamma}$ by $\overline{\gamma}+1$.
\end{proof}

\begin{theorem}\label{d3}
If $s$ is divisible by~$3$, then for any coset~$\mathcal{A}$ in~$\CCo$ of size~$s$, there exist $A,B\in \RR$ such that $j(-A-B)=2^\delta j(B)=2^{2\delta}j(A)$ or $j(A+B)=2^\delta j(B)=2^{2\delta}j(A)$, where $\delta=s/3$ and $j(A)=j\in \mathcal{A}$.
\end{theorem}
\begin{proof}
When $s$ is divisible by~$3$, for a given coset~$\mathcal{A}$ in~$\CCo$ of size $s$, we need to prove that there exist $j\in \mathcal{A}$ and $x_2 \in \TT$ such that $A=1+2y$, $B=x_2(1+2y^{2^\delta})$ and $A+B=u_0(1+2y^{2^{2\delta}})$ or $-A-B=u_0(1+2y^{2^{2\delta}})$, where $y=\xi^j$ and $u_0=1+x_2+2\sqrt{x_2}$ from Yamada's formula (Lemma~\ref{Yamada}).
 Similar to the previous discussion in Theorem~\ref{nd3}, we need to prove that there exist $j\in \mathcal{A}$ and $\overline{y}=\overline{\xi}^j\in F_{2^\Delta}$ such that
\begin{IEEEeqnarray*}{rCl}
  \Tr\big((\overline{y}^{2^\delta}\!+\overline{y}^{2^{2\delta}})(\overline{y}+\overline{y}^{2^{2\delta}})\big)&=&0 \quad
\mbox{or} \\ \Tr\big((\overline{y}^{2^\delta}\!+1+\overline{y}^{2^{2\delta}})(\overline{y}+1+\overline{y}^{2^{2\delta}})\big)&=&0.
\end{IEEEeqnarray*}
By Lemma~\ref{solution1},
\begin{multline*}
 \Tr\big((\overline{y}^{2^\delta}\!+1+\overline{y}^{2^{2\delta}})(\overline{y}+1+\overline{y}^{2^{2\delta}})\big)
 \\
 =\Tr\big(
 (\overline{y}^{2^\delta}\!\!+\overline{y}^{2^{2\delta}})
 (\overline{y}+\overline{y}^{2^{2\delta}})
 +(\overline{y}^{2^\delta}\!\! +\overline{y}^{2^{2\delta}})
 +(\overline{y}+\overline{y}^{2^{2\delta}})
 +1\big)
 \\
 =\Tr\big((\overline{y}^{2^\delta}\!+\overline{y}^{2^{2\delta}})(\overline{y}+\overline{y}^{2^{2\delta}})\big)+1.
\end{multline*}
Hence, one of the equations above can hold.
\end{proof}

The main result of the current section is the following theorem.
\begin{theorem}\label{th:partition}
For every odd $\Delta$,
there is a partition
of
$\RR\backslash\{0\}$
into
$2^\Delta - 1$
triples of form
$\{ a, 2a, 3a\}$
and $\frac{(2^\Delta-1)(2^\Delta-2)}{6}$  sixtuples of form
$\{ \pm a, \pm b, \pm(a+b) \}$
such that the partition is invariant
under the multiplication by $\xi$
and, if $\Delta\not\equiv 0 \bmod 3$, under the action of
the automorphisms of
$\GR(4^\Delta)$.
\end{theorem}
\begin{proof}
It is easy to see that $(\TT \cup 2\TT \cup 3\TT)\backslash\{0\}$ can partition into $2^\Delta - 1$ triples $\{ \xi^i, 2\xi^i, 3\xi^i\}$, $i=0,\ldots,2^\Delta - 2$, which is invariant under the multiplication by $\xi$.

It remains to partition $\RR \backslash(\TT \cup 2\TT \cup 3\TT)$ into
$(2^\Delta - 1)(2^\Delta - 2)/6$ sixtuples of the required form.
We first prove the following:
\begin{itemize}
 \item[(*)] \emph{There are $(2^\Delta - 2)/6$ pairs $(A,B)$,
 $A,B \in \RR \backslash(\TT \cup 2\TT \cup 3\TT)$ such that
 every $j$ in $\{1,\ldots,2^\Delta - 2\}$
 is uniquely represented as
 $j(A)$, $j(B)$, $j(-A-B)$,
 $j(-A)$, $j(-B)$, or $j(A+B)$ for one of these pairs.}
\end{itemize}

To prove (*), we split the set $\{1,\ldots,2^\Delta - 2\}$ into the sets
$\bigcup_{\mathcal{A} \in \Cos(s)} \mathcal{A}$ and for each $s$ construct $N_s\cdot s/6$ pairs
that cover the corresponding set in the same sense
as in (*). This step is divided into two subcases,
based on $3$ divides $s$ or not.

\begin{itemize}
 \item[(i)] \emph{If $s>1$ is not divisible by $3$,
 then there are $N_s\cdot s/6$ pairs $(A,B)$,
 $A,B \in \RR \backslash(\TT \cup 2\TT \cup 3\TT)$ such that
 every $j$ in {$\bigcup_{\mathcal{A} \in \Cos(s)} \mathcal{A}$}
 is uniquely represented as
 $j(A)$, $j(B)$, $j(-A-B)$,
 $j(-A)$, $j(-B)$, or $j(A+B)$ for one of these pairs.}
 \item \emph{Proof of}~(i). By Theorem \ref{nd3}, there are
$A_1$, $B_1$, \ldots, $A_k$, $B_k$,
where $k=N_s/6$,
such that $j(A_i)$, $j(B_i)$, $j(-A_i-B_i)$,
$j(-A_i)$, $j(-B_i)$, $j(A_i+B_i)$, $i=1,\ldots,k$, represent all cosets from $\CCo(s)$. In other words, $\bigcup_{\mathcal{A} \in \Cos(s)} \mathcal{A}$ can split into such sets $\pm \mathcal{A}_i\cup \pm \mathcal{B}_i\cup \pm \mathcal{C}_i$, where $i=1,\ldots,k$. Without loss of generality, suppose $j(A_i)\in \mathcal{A}_i$, $j(B_i)\in \mathcal{B}_i$ and $j(A_i+B_i)\in \mathcal{C}_i$, where $i=1,\ldots,k$. Let $f$ be the automorphism of $\GR(4^\Delta)$ defined in Section~\ref{s2.1}.
Denote $f^l(A_i)$, $f^l(B_i)$ by $A^{(l)}_i$, $B^{(l)}_i$,  where $l=0,1,2,\ldots,s-1$. Since $f^l$ is a homomorphism,
we have
\begin{IEEEeqnarray*}{c}
 j(A^{(l)}_i)=2^lj(A_i),\quad j(B^{(l)}_i)=2^lj(B_i), \quad \mbox{and}\\
j(A^{(l)}_i+B^{(l)}_i)=j((A_i+B_i)^{(l)})=2^lj(A_i+B_i).$$
\end{IEEEeqnarray*}
It is easy to see that
\begin{IEEEeqnarray*}{rClCrCl}
 2^lj(A_i)     &\in& \mathcal{A}_i,& \quad & 2^lj(-A_i)     &\in& -\mathcal{A}_i, \\
2^lj(B_i)      &\in& \mathcal{B}_i,& \quad & 2^lj(-B_i)     &\in& -\mathcal{B}_i, \\
2^lj(A_i+B_i)  &\in& \mathcal{C}_i,& \quad & 2^lj(-A_i-B_i) &\in& -\mathcal{C}_i.
\end{IEEEeqnarray*}
More precisely,
\begin{IEEEeqnarray*}{rClCrCl}
\mathcal{A}_i&=&\{2^lj(A_i)\}_l,     & \quad &  -\mathcal{A}_i&=&\{2^lj(-A_i)\}_l,   \\
\mathcal{B}_i&=&\{2^lj(B_i)\}_l,     & \quad &  -\mathcal{B}_i&=&\{2^lj(-B_i)\}_l,   \\
\mathcal{C}_i&=&\{2^lj(A_i+B_i)\}_l, & \quad &  -\mathcal{C}_i&=&\{2^lj(-A_i-B_i)\}_l
\end{IEEEeqnarray*}
with $l=0,1,2,\ldots,s-1$. It is not difficult to find that
$\bigcup_{\mathcal{A} \in \Cos(s)} \mathcal{A}$ can be expressed as
\begin{IEEEeqnarray*}{rl}
\bigcup_{i=1}^{k}
\bigcup_{l=0}^{s-1}
\Big\{\, &
j\big(A^{(l)}_i\big),\ j\big(B^{(l)}_i\big),\ j\big(A^{(l)}_i+B^{(l)}_i\big),
\\ &
j\big(-A^{(l)}_i\big),\ j\big(-B^{(l)}_i\big),\ j\big(-A^{(l)}_i-B^{(l)}_i\big)
\,\Big\}.
\end{IEEEeqnarray*}
Hence, these pairs $(A^{(l)}_i,B^{(l)}_i)$, where $1\leq i\leq k$, $0\leq l\leq s-1$, satisfy the conclusion of~(i).
\end{itemize}
\begin{itemize}
 \item[(ii)] \emph{If $s$ is  divisible by $3$,
 then for every $\mathcal{A}$ in $\Cos(s) $ there are $s/3$ pairs $(A,B)$,
 $A,B \in \RR \backslash(\TT \cup 2\TT \cup 3\TT)$ such that
 every $j$ in $\pm\mathcal{A}$
 is uniquely represented as
 $j(A)$, $j(B)$, $j(-A-B)$,
$j(-A)$, $j(-B)$, or $j(A+B)$  for one of these pairs.}
 \item[] \emph{Proof of}~(ii). By Theorem~\ref{d3}, for any $\mathcal{A}$ in $\Cos(s)$, there exist $A,B\in \RR$ such that
 \begin{IEEEeqnarray}{rCcCl}
              j(-A-B) & = & 2^\delta j(B) & = & 2^{2\delta}j(A)  \label{eq--}  \\
\mbox{or}\quad j(A+B) & = & 2^\delta j(B) & = & 2^{2\delta}j(A), \label{eq++}
 \end{IEEEeqnarray}
 where $\delta=s/3$ and $j(A)=j\in \mathcal{A}$.
 We suppose \eqref{eq--} holds
 (case~\eqref{eq++} in similar). Let $f$ be the automorphism of $\GR(4^\Delta)$ defined in
 Section~\ref{s2.1}. Denote $f^l(A)$ by $A^{(l)}$, where $l=0,1,2,\ldots,s-1$. Since $f^l$ is a homomorphism, we have
 \begin{IEEEeqnarray*}{rClCrCl}
 j(A^{(l)}) &=&2^lj(A), & &
 j(-A^{(l)})&=&2^lj(-A), \\
 j(B^{(l)}) &=&2^{\delta+l}j(A), &\  &
 j(-B^{(l)})&=&2^{\delta+l}j(-A),   \\
 j(-A^{(l)}-B^{(l)})&=&2^{2\delta+l}j(A), &\  &\\ &&&&
  \makebox[0mm][r]{and \ $j(A^{(l)}+B^{(l)})$} &=&2^{2\delta+l}j(-A).
 \end{IEEEeqnarray*}
 For such $A$, we have
 \begin{IEEEeqnarray*} {rCl}
  \mathcal{A} &=& \{2^lj(A): 0\leq l\leq s-1\} \\
    &=& \bigcup_{i=0,1,2}\big\{2^{l+i\delta}j(A): 0\leq l\leq \delta-1\big\},
 \end{IEEEeqnarray*}
     and (see also  Lemma~\ref{j(-A)})
     \begin{IEEEeqnarray*}{rCl}
       -\mathcal{A} &=&   \{2^lj(-A): 0\leq l\leq s-1  \} \\
        &=& \bigcup_{i=0,1,2} \big \{2^{l+i\delta}j(-A): 0\leq l\leq \delta-1\big\}.
     \end{IEEEeqnarray*}
 It is easy to see now that
      $$\mathcal{A}=\bigcup_{0\leq l\leq \delta-1}\big\{j(A^{(l)}), j(B^{(l)}), j(-A^{(l)}-B^{(l)})\big\},$$
      and
      $$-\mathcal{A}=\bigcup_{0\leq l\leq \delta-1}\big\{j(-A^{(l)}), j(-B^{(l)}), j(A^{(l)}+B^{(l)})\big\}.$$
Therefore, the pairs $(A^{(l)},B^{(l)})$, where $0\leq l\leq \delta-1$,  satisfy the conclusion of~(ii).
\end{itemize}
Since $\{1,\ldots,2^\Delta - 2\}$ is divided into the sets considered
in (i) and (ii), the proof of (*) is completed. Then, for such $(2^\Delta-2)/6$ pairs $(A,B)$, the pairs $(xA,xB)$ can do the same things, where $x\in \TT\setminus \{0\}$.
Therefore,  $$\big\{{\pm x}A_i,\,\pm xB_i,\,\pm x(A_i+B_i)\big\}_{1\leq i\leq (2^\Delta-2)/6,\  x\in \TT\setminus \{0\}}$$ coincides with $\RR \backslash(\TT \cup 2\TT \cup 3\TT)$.
It remains to observe that the following
facts which guarantee the required symmetries
of the constructed partition.
At first, $\TT\setminus \{0\}$
is invariant under the multiplication by $\xi$.
At second, if $s>1$ is not divisible by $3$,
then all nonempty $\Cos(s)$ fall into Case~(i),
and by the construction there
the resulting partition is invariant
under the ring automorphisms. At last, if $s$ is divisible by $3$, then every $\pm \mathcal{A}$ from $\Cos(s)$ fall into Case~(ii), which is also invariant
under the ring automorphisms.
\end{proof}

\section[Quasi-cyclic 1-perfect codes]{Quasi-cyclic 1-perfect codes}\label{s4}
By Lemma~\ref{l:Aperf}, the partition constructed in the previous section
corresponds to a $1$-perfect code in
$D\big((2^\Delta-1)\frac{2^\Delta-2}{6},(2^\Delta-1)\big)$.
The set of columns of the check matrix is invariant
under multiplication by~$\xi$,
and this multiplication
induces a permutation of the coordinates
consisting of cycles
of order~$(2^\Delta-1)$.
\begin{corollary}\label{c:main}
Let $\Delta$ be odd.
Let $\xi$ be a root of any basic primitive irreducible polynomial with degree~$\Delta$ over~$\ZZ_4$.
Let $H$ be the $\Delta\times \frac{2^{2\Delta}-1}{3}$ matrix over~$\ZZ_4$ consisting of $(2^\Delta-1)\frac{(2^\Delta-2)}{6}$ pairs of columns
$(a,b)$, where $\{ \pm a, \pm b, \pm(a+b) \}$
are sixtuples from Theorem~\ref{th:partition},
in the left part and $2^\Delta-1$ columns $\xi^i$, $0\leq i \leq 2^\Delta-2$,
in the right part.
The following holds:
\begin{itemize}
 \item the code~$C$ defined by the check matrix~$H$ is a $1$-perfect code in $D\big((2^\Delta-1)\frac{2^\Delta-2}{6},(2^\Delta-1)\big)$;
 \item
the coordinate permutation that corresponds to multiplying each column by~$\xi$
is an automorphism of the code;
\item if $\Delta$ is not divisible
by~$3$, then the coordinate permutation
that corresponds to a ring automorphism
of~$\GR(4^\Delta)$.
\end{itemize}
\end{corollary}

The proof of the following proposition is the same as that in \cite[Sect.\,4.4]{SHK:additive}.
We omit the complete proof,
only mentioning the idea.
The invariant that distinguishes  the codes
is the number of codewords of weight~$3$
with all three nonzero positions in the right part of the codeword.
Every additive $1$-perfect code
in $D\big((2^\Delta-1)\frac{2^\Delta-2}{6},0+(2^\Delta-1)\big)$ has
$(2^\Delta-1)(2^\Delta-2)/6$ such codewords of order~$2$. The quasicyclic codes
have no such codewords or order~$4$, while the codes constructed recursively in
\cite[Sect.\,4.4]{SHK:additive} have.
\begin{proposition}
Each of the codes in
$D\big( (2^\Delta-1)\frac{2^\Delta-2}{6},0+(2^\Delta-1) \big)$ constructed above is not equivalent to any code constructed in \cite[Sect.\,3.3]{SHK:additive}, with the corresponding parameters.
\end{proposition}

\section{Concluding remarks}\label{s5}
In this paper, we have shown how to construct
additive $1$-perfect codes in a Doob graph
that are invariant under a quasicyclic
permutation of coordinates of order $2^\Delta-1$.
  If $\Delta$ is not divisible by~$3$,
 then the set of columns of the check matrix
 in Corollary~\ref{c:main}
 is also invariant under the ring automorphism of order~$\Delta$. This guarantees that the code is even more symmetric
 than is declared in Corollary~\ref{c:main}. For example, since $5$ and~$2^5-1$
 are coprime, the code in $D(155,31)$ has a permutation automorphism of order
 $155=5\cdot(2^5-1)$, which cyclically permutes each of the two groups of coordinates.
 Such a code can be called \emph{cyclic}, in the sense of~\cite{BorFer:1cyclic}.
 The problem of existence of cyclic $1$-perfect codes in
 $D\big((2^\Delta-1)\frac{2^\Delta-2}{6},(2^\Delta-1)\big)$ for $\Delta\ge 7$ remains open.

Another interesting direction is to further study partitions of
modules over~$\ZZ_4$ like the partitions we constructed in Section~\ref{s:part}, finding their place in discrete geometry,
generalizing to other rings, and seeking their applications in coding theory.
Note that, as shown in~\cite{HerSho71}, additive $1$-perfect codes over non-prime finite fields are related to spreads,
and this relation is similar as the relation
between perfect codes in
Doob graph and partitions of $\GR(4^\Delta)$
considered in the current paper.
Spreads are widely studied objects~\cite{Johnson:spreads},
and we believe that different generalizations to finite rings,
including the one considered in the current paper, deserve attention.

\end{document}